\documentclass{amsart}
 \usepackage{amssymb}
\usepackage{amscd}
\input xy
\xyoption{all}
\newtheorem{theorem}{Theorem}[section]
\newtheorem{lemma}[theorem]{Lemma}

\theoremstyle{definition}
\newtheorem{definition}[theorem]{Definition}
\newtheorem{example}[theorem]{Example}

\theoremstyle{remark}
\newtheorem{remark}[theorem]{Remark}

\theoremstyle{notation}
\newtheorem{notation}[theorem]{Notation}

\theoremstyle{proposition}
\newtheorem{proposition}[theorem]{Proposition}

\theoremstyle{corollary}
\newtheorem{corollary}[theorem]{Corollary}

\numberwithin{equation}{section}

\begin{document}

\title{Cohomology of restricted Lie-Rinehart algebras and the Brauer group}


\author{Ioannis Dokas}
\address{}
\email{dokas@ucy.ac.cy}

\subjclass{17B63, 17A32, 18G60, 18G15}



\keywords{Restricted Lie algebra, Lie-Rinehart algebra, Brauer
group, Quillen-Barr-Beck cohomology.}

\begin{abstract}
We give an interpretation of the Brauer group of a purely
inseparable extension of exponent $1$, in terms of restricted
Lie-Rinehart cohomology. In particular, we define and study the
category $p$-$\rm{LR}(A)$ of restricted Lie-Rinehart algebras over a
commutative algebra $A$. We define cotriple cohomology groups
$H_{p-LR}(L,M)$ for $L\in p$-$\rm{LR}(A)$ and $M$ a Beck $L$-module.
We classify restricted Lie-Rinehart extensions. Thus, we obtain a
classification theorem for regular extensions considered by
Hoshschild.
\end{abstract}

\maketitle





\section*{\textbf{Introduction}}

In classical theory of simple algebras it is known that if $E/F$ is
a Galois extension of fields, then the Brauer group
$\mathcal{B}_{F}^{E}$ of this extension is isomorphic with the group
of equivalence classes of group extensions of the Galois group
$Gal(E/F)$ by the multiplicative group $F^{*}$ of $F$. Moreover, we
have an isomorphism of groups relating the Galois cohomology and the
Brauer group:
\begin{equation}
\mathcal{B}_{F}^{E}\simeq H^{2}(Gal(E/F),F^{*})
\end{equation}
In the context of purely inseparable extensions of exponent $1$
there is a Galois theory due to Jacobson (see \cite{Jac2}). The role
of Galois group is now played by the group of derivations. Purely
inseparable extensions occur naturally in algebraic geometry. In
particular, such extensions appear in the theory of elliptic curves
in prime characteristic. If $k$ is a field such that $char\,k=p$ and
$V$ an algebraic variety over $k$ of dimension greater than $0$,
then the function field $k(V)$ is a purely inseparable extension
over the subfield $k(V)^{p}$ of $pth$ powers.

Hochschild in \cite{Hoch2} proves that the Brauer group
$\mathcal{B}_{k}^{K}$ of purely inseparable extension $K/k$ of
exponent $1$, is isomorphic with the set of equivalence classes of
\textit{regular extensions} of restricted Lie algebras of
$Der_{k}(K)$ by $K$. We remark that the second Hochschild cohomology
group $H_{Hoch}^{2}(L,M)$ where $L$ is a restricted Lie algebra and
$M$ a restricted Lie -module classifies abelian extensions such that
$M^{[p]}=0$. This implies that, there is not an isomorphism between
the Brauer group $\mathcal{B}_{k}^{K}$ and a subgroup of the
cohomology group $H_{Hoch}^{2}(Der_{k}(K),K)$. This remark has been
the motivation to undertake this research. In order to classify
regular extensions and obtain the analogue to $(0.1)$ cohomological
interpretation of the Brauer group, we are led to define and study
the category of restricted Lie-Rinehart algebras. The concept of
Lie-Rinehart algebra is the algebraic counterpart of the notion of
Lie algebroid (see \cite{Macke}). It seems that the notion of
Lie-Rinehart algebra appears first under the name
\textit{pseudo-alg\`{e}bre de Lie} in the paper \cite{He} of Herz.
Also, the notion has been examined by Palais under the name
\textit{d-ring}. The first thorough study of the notion has been
done by Rinehart in \cite{Rin}. Rinehart is the first who defined
cohomology groups for the category of Lie-Rinehart algebras with
coefficients in a Lie-Rinehart module and further developments has
been done by Huebschmann in \cite{Hue}. Moreover, cotriple
cohomology for the category of Lie-Rinehart algebras has been
defined in \cite{Ca} by Casas- Lada- Pirashvili. Besides, the notion
of Lie-Rinehart algebra is closely related to the notion of Poisson
algebra. Loday-Vallette in \cite{LV} remarked that a Lie-Rinehart
algebra is a Poisson algebra. Thus, all the constructions and
properties of Poisson algebras apply to Lie-Rinehart algebras.

In Section $1$ we define the category $p$-$\rm{LR}(A)$ of restricted
Lie-Rinehart algebras over a commutative algebra $A$. We give
examples which occur naturally. In Section $2$ we introduce the
notion of restricted enveloping algebra and restricted Lie-Rinehart
module. We prove a Poincar\'{e}-Birkhoff-Witt type theorem for the
category of restricted Lie-Rinehart algebras. There is a notion of
free Lie-Rinehart algebra made explicit by Casas-Pirashvili (see
\cite{Ca}) and Kapranov (see \cite{Kap}). We extend this notion for
the category $p$-$\rm{LR}(A)$ and we construct free functor left
adjoint to the forgetful functor. In Section $3$ following the
general scheme of Quillen-Barr-Beck cohomology theory for universal
algebras, we determine the Beck modules and Beck derivations. In
Section $4$ we define cohomology groups $H^{*}_{p-LR(A)}(L,M)$ for
$L\in p$-$\rm{LR}(A)$ and $M$ a Beck $L$-module. We prove that
Quillen-Barr-Beck cohomology for $p$-$\rm{LR}(A)$ classifies
extensions of restricted Lie-Rinehart algebras. \textit{Regular
extensions} considered by Hochschild are restricted Lie-Rinehart
extensions. As a consequence in Section $5$, we prove that if $K/k$
is purely inseparable extension of exponent $1$, there is an
isomorphism of groups
$$\mathcal{B}_{k}^{K}\simeq H_{p-LR}^{1}(Der_{k}(K),K)$$

\section{Restricted Lie-Rinehart algebras}

In many cases when we study Lie algebras over a field of prime
characteristic we are led to consider a richer structure than an
ordinary Lie algebra. Indeed the notion of a Lie algebra has to be
replaced by the notion of restricted Lie algebra introduced by N.
Jacobson in \cite{Jac}. Let us recall the definition.

\begin{definition}
A \textit{restricted Lie algebra} $(L,(-)^{[p_{L}]})$ over a field
$k$ of characteristic $p\neq 0$ is a Lie algebra $L$ over $k$
together with a map $(-)^{[p_{L}]}: L \rightarrow L$ called the
$p$-map such that the following relations hold:

\begin{align}
(\alpha x)^{[p_{L}]} &=  \alpha^{p}\;x^{[p_{L}]}\\
[x,y^{[p_{L}]}]     &= [x,\underbrace{y],y],\cdots,y}_{p}] \\
(x+y)^{[p_{L}]} &=
x^{[p_{L}]}+y^{[p_{L}]}+\sum_{i=1}^{p-1}s_{i}(x,y)
\end{align}
where $i s_{i}(x,y)$ is the coefficient of $\lambda^{i-1}$ in
$ad^{p-1}_{\lambda x+y}(x)$, where $ad_{x}: L\rightarrow L$ denotes
the adjoint representation given by $ad_{x}(y):=[y,x]$ and $x,y \in
L,\; \alpha \in k$. We denote by $RLie$ the category of restricted
Lie algebras over $k$.

A Lie-module $A$ over a restricted Lie algebra $(L, (-)^{[p]})$ is
called \textit{restricted} if $x^{[p]}m=(\underbrace{x\cdots
(x(x}_{p}m)\cdots)$
\end{definition}

\begin{example}
Let $R$ be any associative algebra over a field $k$ with
characteristic $p\neq 0$. We denote by $R_{Lie}$ the induced Lie
algebra with the bracket given by $[x,y]:=xy-yx$, for all $x,y\in
R$. Then $(R,(-)^{p})$ is a restricted Lie algebra where $(-)^{p}$
is the Frobenious map given by $x\mapsto x^{p}$. Thus, there is a
functor $(-)_{RLie}: \rm{As} \rightarrow \rm{RLie}$ from the
category of associative algebras to the category of restricted Lie
algebras.
\end{example}

\begin{proposition}
Let $L$ be a restricted Lie algebra and $A$ a restricted $L$-module. If we
denote by $A\oplus L$ the direct sum of the underlying vectors
spaces of $A$ and $L$ then $A\oplus L$ is endowed with the
structure of restricted Lie-algebra.
\end{proposition}

\begin{proof}
It is well known that $A\oplus L$ is endowed with the structure of a
Lie algebra with bracket given by:
$$[a+X,b+Y]=(X(b)-Y(a))+[X,Y]$$ for any $a,b\in A$ and $X,Y\in L$.

Moreover, we have:
\begin{align*}
[a+X,\underbrace{Y,]\cdots,Y}_{p}\;] &=-(\underbrace{Y(Y(\dots
(Y}_{p}(a)))) +[X,\underbrace{Y],\cdots
,Y}_{p}\;]\\
 &=-(\underbrace{Y(Y(\dots (Y}_{p}(a))))+[X,Y^{[p]}]\\
&=[a+X,Y^{[p]}]
\end{align*}
Besides,
$$[a+X,\underbrace{b]\cdots,b}_{p}\;]=0$$
Therefore, from Jacobson's theorem there exists a unique $p$-map
$$(-)^{[p]}: A\oplus L \rightarrow A\oplus L$$ which extents the
$p$-map on $L$ and such that $a^{[p]}=0$ for all $a\in A$. In
particular we see that the $p$-map on $A\oplus L$ is given by:
$$(a+X)^{[p]}=(\underbrace{X(X(\dots (X}_{p-1}(a))))+X^{[p]}$$
We denote this restricted Lie algebra structure on $A\oplus L$ by
$A\rtimes L$.
\end{proof}

\begin{remark}
Let $A$ be a restricted $L$-module. We consider the invariants for
the Lie action $A^{L}=\{a\in A : la=0,\;\; \text{for all}\; l\in
L,\;a\in A\}$. If $f: A\rightarrow A^{L}$ is a $p$-semi linear map,
then it is easily seen that $A\oplus L$ is a restricted Lie algebra
with $p$-map given by $(a+X)^{[p]}:=(\underbrace{X(X(\dots
(X}_{p-1}(a))))+X^{[p]}+f(a)$. We denote this restricted Lie algebra
by $A\rtimes_{f} L$.
\end{remark}

Let $A$ be a commutative algebra over a field $k$. A $k$-linear map
$D: A\rightarrow A$ is called a $k$-derivation if $$D(ab)=aD(b)+D(a)b$$
 Let $\rm{Der_{k}}(A)$ be the set of $k$-derivations of $A$. It is well known
 that if $D,D'\in \rm{Der}(A)$, then $[D,D']:=DD'-D'D$ is a
 derivation. Thus, $(\rm{Der}(A),[-,-])$ is a Lie algebra. Moreover,
 if $D\in \rm{Der_{k}}(A)$ and $a,x\in A$ then $aD: A\rightarrow A$, given by:
 $(aD)(x):=aD(x)$ is a derivation. Therefore, $\rm{Der_{k}}(A)$ has the
 structure of an $A$-module. Besides the following relation holds:

 $$[D,aD']=a[D,D']+D(a)D'$$

The structure on $\rm{Der_{k}}(A)$ described above is the prototype
example of the notion of Lie-Rinehart algebra. Let us recall the
definition.

\begin{definition}
A \textit{Lie-Rinehart} algebra over $A$, or $(k-A)$-Lie algebra, is
a pair $(A,L)$ where, $A$ is a commutative algebra over $k$, $L$ is
a Lie algebra over $k$ equipped with the structure of an $A$-module
together with a map called \textit{anchor} $\alpha: L\rightarrow
\rm{Der_{k}}(A)$ which is an $A$-module and a Lie algebra morphism
such that:
$$[X,aY]=a[X,Y]+\alpha(X)(a)Y$$
for all $a\in A$ and $X,Y \in L$.
\end{definition}

In order to simplify the notation we denote $\alpha(X)(a)$ by
$X(a)$. Moreover, we denote by $\rm{LR}(A)$ the category of
Lie-Rinehart algebras over $A$.

\begin{example}
We easily see that $(A,\rm{Der_{k}}(A))$ is a Lie-Rinehart algebra
with anchor map $id: \rm{Der_{k}}(A)\rightarrow \rm{Der_{k}}(A)$.
\end{example}

Suppose now that the ground field $k$ is a field of characteristic
$p\neq 0$. Let $D\in \rm{Der_{k}}(A)$, from Leibniz rule for all
$a,b \in A$ we have
$$D^{p}(ab)=\sum_{i=0}^{i=p}\binom{p}{i}D^{i}(a)D^{p-i}(b)$$ Since
$char\;k=p$ we get: $$D^{p}(ab)=aD^{p}(b)+D^{p}(a)b$$ Therefore,
$D^{p}$ is a derivation. In other words $\rm{Der_{k}}(A)$ is
equipped with the structure of restricted Lie algebra. Moreover, by
Hochschild's Lemma $1$ in \cite{Hoch2} we get the relation:
\begin{equation}
(aD)^{p}=a^{p}D^{p}+(aD)^{p-1}(a)D
\end{equation}
Therefore, we see in prime characteristic that the set of derivations
$\rm{Der_{k}}(A)$ has a richer structure than just a Lie-Rinehart
algebra structure. We are naturally led to the following definition
of restricted Lie-Rinehart algebra.

From now on we fix a field $k$ of characteristic $p\neq 0$.
\begin{definition}
A \textit{restricted Lie-Rinehart} algebra $(A,L,(-)^{[p]})$ over a
commutative $k$-algebra $A$, is a Lie-Rinehart algebra over $A$ such
that: $(L,(-)^{[p]})$ is a restricted Lie algebra over $k$, the
anchor map is a restricted Lie homomorphism, and the following
relation holds:

\begin{equation}
(aX)^{[p]}=a^{p}X^{[p]}+ (aX)^{p-1}(a)X
\end{equation}
for all $a\in A$ and $X\in L.$
\end{definition}

Let $(A,L,(-)^{[p]})$ and $(A',L',(-)^{[p]})$ be restricted
Lie-Rinehart algebras. Then a Lie-Rinehart morphism $(\xi,f):
(A,L,(-)^{[p]})\rightarrow (A',L',(-)^{[p]})$ is called restricted
Lie-Rinehart morphism if $f$ is a restricted Lie morphism, namely:
$f(x^{[p]})=f(x)^{[p]}$ for all $x\in L$. We will denote by
$p-\rm{LR}$ the category of restricted Lie-Rinehart algebras.

\begin{example}
As we have seen if $A$ is a commutative algebra over $k$, then
$Der_{k}(A)$ is a restricted Lie-Rinehart algebra.
\end{example}

\begin{example}
Any restricted Lie algebra over $k$ is a restricted Lie-Rinehart
algebra $(A,L,(-)^{[p]})$, where $A=k$.
\end{example}

 The structure of restricted Lie Rinehart algebra appears in
Jacobson-Galois theory of purely inseparable extensions of exponent
$1$.

\begin{example}
 Let $K/k$ be a purely inseparable field extension of
exponent $1$. Then, there is a one-to-one correspondence between
intermediate fields and restricted Lie Rinehart sub-algebras of the
restricted Lie Rinehart algebra $Der_{k}(K)$ over $K$ (see
\cite{Jac2}).
\end{example}

The Jacobson-Galois correspondence for purely inseparable extensions
of exponent $1$ has been used  by several authors  in order to study
isogenies of algebraic groups especially of abelian varietes. The
notion of restricted Lie-Rinehart appears in this study.

\begin{example}
Let $G$ be an algebraic group over $k$, and $K$ be the field of
rationales functions of $G$. Then, the $K$-space of derivations
$D_{k}(K)$ is a restricted Lie-Rinehart algebra over $K$  (see
\cite{Ser}).
\end{example}

The structure of restricted Lie Rinehart algebras emerges also in
connection with theory of characteristic classes.

\begin{example}
Let $k\subset K$ be fields of characteristic $p\neq 0$. Then,
Maakestad considers the category $Lie_{K/k}$ (see \cite{Ma}) whose
objects are restricted Lie Rinehart sub-algebras of the restricted
Lie Rinehart algebra $Der_{k}(K)$ over $K$. Moreover, Maakestad
construct a contravariant functor which associates $g\in Lie_{K/k}$
to the Grothendiek ring $K_{0}(g)$ of the category of
$g$-connections. (for details see Theorem $3.2$ in \cite{Ma}).
\end{example}

\begin{example}
Let $P$ be a Poisson algebra over $k$. If we denote by $$C:=\{c\in
P\;\; |\; [c,-]=0\}$$ then we see that $C$ is closed under the
commutative and Lie bracket. Thus, $C$ is a Poisson subalgebra of
$P$. We call a \text{Poisson derivation} a $k$-linear map
$D:P\rightarrow P$ which is at the same time a derivation with
respect to commutative and Lie product. We can easily see that,
$\rm{Der_{k}}(P)$ has the structure of Lie $k$-algebra. Moreover, by
Leibniz rule we can see that $D^{p}$ is a derivation with respect to
commutative and the Lie product. Thus, $\rm{Der_{k}}(P)$ has the
structure of restricted Lie $k$-algebra. Besides, if $c\in C$ and
$D\in \rm{Der_{k}}(P)$ then $cD\in \rm{Der_{k}}(P)$. Moreover, by
Hochchild's Lemma $1$ (see \cite{Hoch2}) we have the relation:
$$(cD)^{p}=c^{p}D^{p}+(cD)^{p-1}(c)D$$ for all $c\in C$
Therefore, $\rm{Der_{k}}(P)$ is a restricted Lie-Rinehart algebra
over $C$.
\end{example}

\section{Restricted enveloping algebras and restricted modules}

Let $(A,L)$ be a Lie-Rinehart algebra. There is a notion of univeral
enveloping algebra $U(A,L)$ of $(A,L)$ defined by Rinehart in
\cite{Rin}. The universal enveloping algebra $U(A,L)$ is an
associative $A$-algebra which verifies the appropriate universal
property (see \cite{Hue}). We recall the definition of the
enveloping associative algebra $U(A,L)$.

The direct sum $A\oplus L$ of the underlying vector spaces has the
structure of $k$-Lie algebra given in the Proposition $1.3$. Let
$(\mathcal{U}(A\oplus L),\iota)$ be the enveloping algebra where
$\iota : A\oplus L \rightarrow \mathcal{U}(A\oplus L)$ is the
canonical embedding. We consider the subalgebra
$\mathcal{U}^{+}(A\oplus L)$ generated by $A\oplus L$. Moreover,
$A\oplus L$ has the structure of an $A$-module via
$$a(a'+X):=aa'+aX$$ for all $a,a'\in A$ and $X\in L$. Then, the
enveloping algebra $U(A,L)$ is defined as the quotient:

$$U(A,L):=\mathcal{U}^{+}(A\oplus L)/<\iota(a)\iota(a'+X)-\iota(a(a'+X))
>$$

The canonical map $\iota_{A}$ is an $A$-algebra homomorphism. The
canonical representation $A\rightarrow End_{k}(A)$ given by the
multiplication is faithful. Thus, by the universal property of
$U(A,L)$ we obtain that the $A$-algebra homomorphism $\iota_{A}:
A\rightarrow U(A,L)$ is injective. The canonical map $\iota_{L}$ is
a Lie algebra homomorphism. Moreover, in $U(A,L)$ the following
relations hold:
$$\iota_{A}(a)\iota_{L}(X)=\iota_{L}(aX),\; and \;\;
[\iota_{L}(X),\iota_{A}(a)]=\iota_{A}(X(a))$$ for all $a\in A$ and
$X\in L$.

The enveloping algebra $U(A,L)$ has a canonical filtration:
$$ A=U_{0}(A,L) \subset U_{1}(A,L)\subset U_{2}(A,L)\cdots $$ where
$U_{n}(A,L)$ is spanned by $A$ and the powers $\iota_{L}(L)^{n}$.
Therefore, we can construct the associated graded algebra given by
$$gr(U(A,L))=\oplus_{n=0}^{\infty}U_{n}(A,L)/U_{n-1}(A,L)$$ where we
set $U_{-1}(A,L)={0}$. We note that $gr(U(A,L))$ is a commutative
$A$-algebra. There is a theorem of Poincare-Birkhoff- Witt type due
to Rinehart. In particular, it is proved (see \cite{Rin},
Theorem $3.1$) that if $L$ is a projective $A$-module and $S_{A}(L)$
denotes the symmetric $A$-algebra on $L$ then the canonical
epimorphism $\theta: S_{A}(L) \rightarrow gr(U(A,L))$ is an
isomorphism of $A$-algebras. Moreover, we obtain that $\iota_{L}:
L\rightarrow U(A,L)$ is injective.

Let $(A,L)$ be a restricted Lie-Rinehart algebra over $A$ and
suppose that $L$ is free as an $A$-module. Let $\{u_{i}, i\in I\}$
be an ordered $A$-basis of $L$. Let $\mathcal{C}(U(A,L))$ denote the
center of $U(A,L)$. Since $L$ is a restricted Lie algebra we obtain:
for all $u_{i}$ there is a $z_{i}\in \mathcal{C}(U(A,L))$ such that
$u_{i}^{p}-u_{i}^{[p]}=z_{i}$.

\begin{theorem}
Let $(A,L)$ be a restricted Lie-Rinehart algebra such that $L$ is
free as an $A$-module. Then the set,
$$B:=\{z_{i_{1}}^{h_{1}}z_{i_{2}}^{h_{2}}\cdots z_{i_{r}}^{h_{r}}
u_{i_{1}}^{k_{1}}u_{i_{2}}^{k_{2}}\cdots u_{i_{r}}^{k_{r}}\}$$ where
$i_{1}<i_{2}<\cdots <i_{r}$, $h_{i}\geq 0$ and $0\leq k_{i}<p$ is an
$A$-basis of $U(A,L).$
\end{theorem}

\begin{proof}
It is proved in Theorem $3.1$ in \cite{Rin} that the standard
monomials of the form $u_{i_{1}}^{s_{1}}u_{i_{2}}^{s_{2}}\cdots
u_{i_{r}}^{s_{r}}$ where $i_{1}<i_{2}<\cdots <i_{r}$ and $s_{i}\geq
0$ form an $A$-basis of $U(A,L)$. Let
$u_{i_{1}}^{s_{1}}u_{i_{2}}^{s_{2}}\cdots u_{i_{r}}^{s_{r}}$ be a
standard monomial and $s=s_{1}+\cdots+s_{r}$. By induction on $s$ we
prove that the set $B$ generates $U(A,L)$. If all $s_{i}$ are such
that $s_{i}<p$ then it is clear. Suppose that there is $s_{i_{j}}>p$
then we have
$$u_{i_{1}}^{s_{1}}u_{i_{2}}^{s_{2}}\cdots
u_{i_{r}}^{s_{r}}=z_{i_{j}}u_{i_{1}}^{s_{1}}u_{i_{2}}^{s_{2}}\cdots
u_{i_{j}}^{s_{i_{j}}-p} \cdots
u_{i_{r}}^{s_{r}}+u_{i_{1}}^{s_{1}}u_{i_{2}}^{s_{2}}\cdots
u_{i_{j}}^{s_{i_{j}}-p}u_{i_{j}}^{[p]} \cdots u_{i_{r}}^{s_{r}}$$

We notice that the terms $u_{i_{1}}^{s_{1}}u_{i_{2}}^{s_{2}}\cdots
u_{i_{j}}^{s_{i_{j}}-p} \cdots u_{i_{r}}^{s_{r}}$ and
$u_{i_{1}}^{s_{1}}u_{i_{2}}^{s_{2}}\cdots
u_{i_{j}}^{s_{i_{j}}-p}u_{i_{j}}^{[p]} \cdots u_{i_{r}}^{s_{r}}$
belong in $U_{s-1}(A,L)$, thus by induction can be written as linear
combination of elements of $B$. Next we prove that the elements of
$B$ are linearly independent. Let
$z_{i_{1}}^{h_{1}}z_{i_{2}}^{h_{2}}\cdots z_{i_{r}}^{h_{r}}
u_{i_{1}}^{k_{1}}u_{i_{2}}^{k_{2}}\cdots u_{i_{r}}^{k_{r}}$ be an
element of $B$. Since $z_{i}=u_{i}^{p}-u_{i}^{[p]}$ we get

\begin{multline}
z_{i_{1}}^{h_{1}}z_{i_{2}}^{h_{2}}\cdots z_{i_{r}}^{h_{r}}
u_{i_{1}}^{k_{1}}u_{i_{2}}^{k_{2}}\cdots u_{i_{r}}^{k_{r}}
=(u_{1}^{p}-u_{1}^{[p]})^{h_{1}}(u_{2}^{p}-u_{2}^{[p]})^{h_{2}}\cdots
(u_{r}^{p}-u_{r}^{[p]})^{h_{r}}
u_{i_{1}}^{k_{1}}u_{i_{2}}^{k_{2}}\cdots u_{i_{r}}^{k_{r}}\\
\equiv u_{i_{1}}^{h_{1}p+k_{1}}u_{i_{2}}^{h_{2}p+k_{2}}\cdots
u_{i_{r}}^{h_{r}p+k_{r}}\mod U_{s-1}(A,L)
\end{multline} where
$s=\sum_{i=1}^{i=r}h_{i}p+k_{i}$ and $i_{1}<i_{2}<\cdots <i_{r}$.

Let $(h_{i_{1}},h_{i_{2}}, \cdots, h_{i_{r}})$,
$(k_{i_{1}},k_{i_{2}}, \cdots, k_{i_{r}})$ and
$(h'_{i_{1}},h'_{i_{2}}, \cdots, h'_{i_{r}})$,
$(k'_{i_{1}},k'_{i_{2}}, \cdots, k'_{i_{r}})$ be sequences such that
$h_{i}p+k_{i}=h'_{i}p+k'_{i}$ for all $i=1,2,\cdots r$. Since
$0<k_{i}<p$ and $0<k'_{i}<p$ it is obliged to have $h_{i}=h'_{i}$
and $k_{i}=k'_{i}$ for all $i=1,2,\cdots r$. Moreover, by
Poincare-Birkhoff-Witt theorem  we have an isomorphism of
$A$-algebras $S_{A}(L)\simeq gr(U(A,L))$. Thus, by the relation
$(2.1)$ follows that the elements of $B$ are linearly independent.
\end{proof}

\begin{proposition}
Let $(A,L)$ be a Lie-Rinehart algebra such that $L$ is free as an
$A$-module. Let $\{u_{i},\;I \}$ be an ordered $A$-basis of $L$. If
there is a map $u_{i}\rightarrow u_{i}^{[p]}$ such that
$ad_{u_{i}}^{p}=ad_{u_{i}^{[p]}}$ for all $i\in I$ then $(A,L)$ can
be equipped with the structure of restricted Lie-Rinehart algebra
with a $p$-map which extends the map $(-)^{[p]}$.
\end{proposition}

\begin{proof}
Let $J$ be the ideal of $U(A,L)$ generated by the
$z_{i}:=u_{i}^{p}-u_{i}^{[p]}$. We consider the associative algebra
$U_{p}(A,L):=U(A,L)/J$. By the previous theorem we get that the elements
of the form
$$u_{i_{1}}^{k_{1}}u_{i_{2}}^{k_{2}}\cdots u_{i_{r}}^{k_{r}}$$ where
$i_{1}<i_{2}<\cdots <i_{r}$, and $0\leq k_{i}<p$ constitute an
$A$-basis for $U_{p}(A,L).$

The canonical $A$-algebra homomorphism
$\iota_{A}: A\rightarrow U(A,L)$ induce an $A$-algebra homomorphism
$i_{A}: A\rightarrow U_{p}(A,L)$. Moreover, the Lie homomorphism
$\iota_{L}: L \rightarrow U(A,L)_{Lie}$ induce a Lie algebra
homomorphism $i_{L}: L\rightarrow U_{p}(A,L)_{Lie}$ which is injective. Moreover, we have
$$(i_{L}(u_{i}))^{p}=u_{i}^{p}+J=u_{i}^{[p]}+J$$
and $(i_{L}(u_{i}))^{p}\in i_{L}(L)$.
Besides,
$$(i_{L}(au_{i}))^{p}=(au_{i})^{p}+J$$ By Hochschild's relation
Lemma $1$ in \cite{Hoch2} we get in $U(A,L)$ the relation
\begin{align*}
(au_{i})^{p}
&=a^{p}u_{i}^{p}+[\underbrace{au_{i},[au_{i},\cdots[au_{i}}_{p-1},a],\cdots]u_{i}\\
&=a^{p}u_{i}^{p}+(au_{i})^{p-1}(a)u_{i}\\
\end{align*}
Therefore,
\begin{align*}
 (i_{L}(au_{i}))^{p} &=a^{p}u_{i}^{p}+(au_{i})^{p-1}(a)u_{i}+J\\
                     &=a^{p}u_{i}^{[p]}+(au_{i})^{p-1}(a)u_{i}+J
 \end{align*}
and $(i_{L}(au_{i}))^{p}\in i_{L}(L)$. Therefore, by relation
$(1.3)$ we get that for all $x\in L$ we have $x^{p}\in i_{L}(L)$ and
$L$ is a restricted Lie $k$-subalgebra of $U_{p}(A,L)_{RLie}$.
Obviously the relation $(1.5)$ of the definition holds and $(A,L)$
is equipped with the structure of a restricted Lie-Rinehart algebra.
\end{proof}

\begin{remark}
Let $(A,L)$ be a Lie-Rinehart algebra such that $L$ is free as an
$A$-module. Let $\{u_{i},\;I \}$ be an ordered $A$-basis of $L$. We
easily see using the relations $1.3$ and $1.4$, that the ideal of
$U(A,L)$ generated by the elements $\{X^{p}-X^{[p]},\;X\in L\}$ is
equal to $J$.
\end{remark}

Next, we introduce the notion of restricted enveloping algebra of a
restricted Lie-Rinehart algebra. Let $(A,L,(-)^{[p]})$ be a
restricted Lie-Rinehart algebra. We define as \textit{restricted
enveloping algebra} $U_{p}(A,L)$ of a restricted Lie-Rinehart
algebra $(A,L,(-)^{[p]})$, the quotient:
$$U_{p}(A,L):=U(A,L)/<X^{[p]}-X^{p} >$$
where, $a\in A$ and $X\in L$.

\begin{remark}
We note that in Proposition $2.2$ is proved that the map $i_{L}: L
\rightarrow U_{p}(A,L)_{RLie}$ is a restricted Lie monomorphism when
$L$ is a free as an $A$-module.
\end{remark}

 The following proposition gives us
the universal property of the restricted enveloping algebra.

\begin{proposition}
Let $B$ be an $A$-algebra such that there is an $A$-algebra
homomorphism $\phi_{A}: A \rightarrow B$ and $\phi_{L}: L\rightarrow
B_{RLie}$ a restricted Lie $k$-homomorphism. If we have:
$$\phi_{A}(a)\phi_{L}(X)=\phi_{L}(aX),\; and \;\;
[\phi_{L}(X),\phi_{A}(a)]=\phi_{A}(X(a))$$ for all $a\in A$ and $X
\in L$ then there exists a unique homomorphism of associative
algebras $\Phi_{p}: U_{p}(A,L) \rightarrow B$ such that $\Phi_{p}
i_{A}=\phi_{A}$ and $\Phi_{p} i_{L}=\phi_{L}$.
\end{proposition}

\begin{proof}
We easily see that the map $f: A\oplus L\rightarrow B_{Lie}$ given
by
$$f(a+X)=\phi_{A}(a)+\phi_{L}(X),$$ for all $a\in A$ and $X\in L$ is a
Lie morphism. Therefore, there is a an algebra morphism $f':
\mathcal{U}^{+}(A\oplus L)\rightarrow B$. Moreover,
\begin{align*}
f'(\iota(a(a'+X))) &=f(aa'+aX)\\
            &=\phi_{A}(aa')+\phi_{L}(aX)\\
            &=\phi_{A}(a)\phi_{A}(a')+\phi_{A}(a)\phi_{L}(X)\\
            &=f(a)f(a'+X)\\
            &=f'(\iota (a))f'(\iota (a'+X))\\
            &=f'(\iota (a)(\iota (a'+X))
\end{align*}
Thus, $f'$ induces an algebra morphism $\Phi: U(A,L)\rightarrow B$.
Since, $\phi_{L}$ is a restricted Lie homomorphism we have
$\Phi(X^{[p]})=\Phi(X^{p})$. Therefore, $\Phi$ induces an algebra
morphism $\Phi_{p}: U_{p}(A,L) \rightarrow B$.
\end{proof}

\begin{remark}
The canonical representation $A\rightarrow End_{k}(A)$ given by the
multiplication, is faithful. By the universal property of
$U_{p}(A,L)$ we get that the $A$-algebra homomorphism $i_{A}:
A\rightarrow U_{p}(A,L)$ is injective.
\end{remark}

\begin{definition}
Let $(A,L,(-)^{[p]})$ be a restricted Lie-Rinehart algebra. A
\textit{restricted Lie-Rinehart} module, is a Lie-Rinehart
$(A-L)$-module $M$ which additionally is a restricted Lie
$L$-module. In other words, a restricted Lie-Rinehart $(A-L)$-module
is a $k$-module $M$ equipped with the structures of an $A$-module
and a Lie $L$-module such that:
\begin{align*}
(aX)m    &=a(Xm)\\
X(am)    &=aX(m)+X(a)m \\
X^{[p]}m &=(\underbrace{X(X(\dots(X}_{p}m))))
\end{align*}
for all $a\in A$, $X\in L$ and $m\in M$.
\end{definition}

Let $(A,L,(-)^{[p]})$ be a restricted Lie-Rinehart algebra. The
category of restricted $(A,L)$-modules is equivalent to the category
of $U_{p}(A,L)$-modules.

\begin{example}
The notion of a restricted Lie-Rinehart module recovers in a
particular case the notion of \textit{regular module} defined in
\cite{Ber}. Namely, if $K/k$ is a purely inseparable extension of
exponent $1$ then a regular module is just a restricted Lie-Rinehart
module over the restricted Lie Rinehart algebra $Der_{k}(K)$.
\end{example}

\begin{example}
Let $\mathbf{g}\in Lie_{K/k}$ be a restricted Lie-Rinehart algebra
(see 1.12 and \cite{Ma}) then a $p$-flat connection is a restricted
Lie-Rinehart $\mathbf{g}$-module.
\end{example}

An important example of Lie algebroid is the \textit{transformation
Lie algebroid} for a differential manifold. There is an algebraic
generalazation of this notion called the \textit{transformation
Lie-Rihart algebra}.

\begin{proposition}
Let $\textbf{g}\in RLie$ be a restricted Lie $k$-algebra and $A$ a
commutative $k$-algebra. If there is a restricted Lie homomorphism
$\delta: \textbf{g}\rightarrow Der(A)$, then the
\textit{transformation} Lie-Rinehart algebra $(A,A\otimes_{k}
\textbf{g})$ can be endowed with the structure of a restricted Lie
-Rinehart algebra.
\end{proposition}

\begin{proof}
The Lie-Rinehart algebra of transformation is a Lie $k$-algebra with
Lie bracket given by:
$$[a\otimes g,a'\otimes g']:=aa'\otimes[g,g']+a\delta(g)(a')\otimes
g'-a'\delta(g')(a)\otimes g$$ for all $a,a'\in A$ and $g,g' \in
\textbf{g}$, and with anchor map $\alpha: A\otimes_{k}
\textbf{g}\rightarrow Der(A)$ given by $\alpha(a\otimes
g)(a')=a\delta(g)(a')$. Let $\{g_{i}, i\in I\}$ be a $k$-basis of
$\textbf{g}$, then the elements $\{1\otimes g_{i}, i\in I\}$ form an
$A$-basis of $A\otimes \textbf{g}$. Let $\tau_{1\otimes
g_{i}},\rho_{1\otimes g_{i}}: A\otimes_{k} \textbf{g} \rightarrow
A\otimes \textbf{g} $ be the $k$-linear maps given by
$\tau_{1\otimes g_{i}}(a\otimes g):=a\otimes [g,g_{i}]$ and
$\rho_{1\otimes g_{i}}(a\otimes g):=\delta(g_{i})(a)\otimes g$
respectively. We note that: $$\tau_{1\otimes g_{i}}\rho_{1\otimes
g_{i}}=\rho_{1\otimes g_{i}}\tau_{1\otimes g_{i}}$$ Since
$char\;k=p$ we have:
\begin{align*}
ad_{1\otimes g_{i}}^{p} &=(\tau_{1\otimes g_{i}}-\rho_{1\otimes g_{i}})^{p}    \\
         &=\sum_{j=0}^{j=p}\binom{p}{j}\tau_{1\otimes g_{i}}^{j}\rho_{1\otimes
         g_{i}}^{p-j}\\
         &=\tau_{1\otimes g_{i}}^{p}-\rho_{1\otimes g_{i}}^{p}\\
         &=ad_{1\otimes g_{i}^{[p]}}
\end{align*}

Therefore, by Proposition $2.2$ above we get that $A\otimes_{k}
\textbf{g}$ can be equipped with the structure of a restricted
Lie-Rinehart algebra and the $p$-map is given by:
$$(a\otimes g)^{[p]}=a^{p}\otimes g^{[p]}-(a\delta(g))^{p-1}(a)\otimes
g$$ By equation $(1.4)$, we see that the anchor map $\alpha$ is
actually a restricted Lie homomorphism.
\end{proof}

 In the next subsection we extend for the category $p-\rm{LR}(A)$ the
 notion of free Lie-Rinehart algebra defined by Casas-Pirashvili in \cite{Ca} and Kapranov in \cite{Kap}.

\subsection{Free restricted Lie-Rinehart algebra}
Let $A$ be a commutative $k$-algebra. We denote by
$\mathrm{Vect}/Der(A)$ the category whose objects are $k$-linear
morphisms $\psi: V \rightarrow Der(A)$ where $V\in \mathrm{Vect}$
and morphisms $f: \psi \rightarrow \psi'$ are $k$-linear morphisms
$f: V\rightarrow V'$ such that $\psi'f=\psi$. We denote by
$\mathcal{V}: p-\mathrm{LR}(A) \rightarrow \mathrm{Vect}/Der(A)$ the
forgetful functor from the category of restricted Lie-Rinehart
$A$-algebras to the category $\mathrm{Vect}/Der(A)$ which assigns a
restricted Lie-Rinehart algebra $L$ over $A$ to $\alpha:
L\rightarrow Der(A)$ the anchor map of $L$.

\begin{proposition}
There is a left adjoint functor $F : \mathrm{Vect}/Der(A)
\rightarrow p-\rm{LR}(A)$ to the functor $\mathcal{V}:
p-\mathrm{LR}(A) \rightarrow \mathrm{Vect}/Der(A)$
$$Hom_{p-\rm{LR}(A)}(F(\psi), L)\simeq
Hom_{\mathrm{Vect}/Der(A)}(\psi, \mathcal{V}(L))$$
\end{proposition}

\begin{proof}
Let $\psi: V\rightarrow Der(A)$. Then, by the universal property of
the free restricted Lie algebra $L_{p}(V)$ generated by $V$ there is
a restricted Lie homomorphism $\Phi: L_{p}(V)\rightarrow Der(A)$
such that $\Phi i_{V}=\psi$, where $i_{V}: V\rightarrow L_{p}(V)$
denotes the canonical map . By Proposition $2.2$, we get that
$A\otimes L_{p}(V)$ is equipped with the structure of a restricted
Lie-Rinehart algebra. Therefore, we construct a functor $F :
\mathrm{Vect}/Der(A) \rightarrow p-\rm{LR}$ which assigns $\psi \in
\mathrm{Vect}/Der(A)$ to $A\otimes_{k} L_{p}(V)$.

Let $f\in Hom_{\mathrm{Vect}/Der(A)}(\psi, \mathcal{V}(L))$. Then,
we have that $\alpha f=\psi$. Moreover, by the universal property of
the free restricted Lie algebra $L_{p}(V)$ generated by $V$, there
is a restricted Lie homomorphism $\phi: L_{p}(V)\rightarrow L$ such
that $\phi i_{V}=f$ and $\Phi=\alpha \phi$. Let $f_{p}: A\otimes_{k}
L_{p}(V) \rightarrow L$ be the homomorphism of restricted
Lie-Rinehart algebras given by $f_{p}(a\otimes x):=a\phi(x)$, where
$a\in A$ and $x\in L_{p}(V)$. Thus, we construct a map $f\mapsto
f_{p}$. Conversely, for $f_{p}\in Hom_{p-\rm{LR}(A)}(F(\psi), L)$ we
consider the $k$-linear map $f: V\rightarrow L$ given by
$f:=f_{p}\bar{i}$ where $\bar{i}: V \rightarrow A\otimes_{k}
L_{p}(V)$ given by $v \mapsto (1\otimes v)$, and $v\in V$. We easily
see that the maps $f \mapsto f_{p}$ and $f_{p} \mapsto f$ are
inverse to each other.
\end{proof}

\section{Beck modules and Beck derivations}

Beck in his desertion (see \cite{Beck}) gave an answer of what
should be the right notion of coefficient module for cohomology. The
notion of Beck-module encompasses for various categories, the usual
known notions of coefficient module for cohomology (see \cite{Bar}).
In this section we determine the category of Beck modules and the
group of Beck derivations (see \cite{Beck}, \cite{BB}) for the
category of restricted Lie-Rinehart algebras $p$-$\rm{LR}(A)$ over a
commutative algebra $A$.

\begin{definition}
Let $L$ be an object in a category $p$-$\rm{LR}(A)$. We denote by
$(p$-$\rm{LR}(A)/L)_{ab}$ the category of abelian group objects of
the comma category $p$-$\rm{LR}(A)/L$ and by $I_{L}:
(p$-$\rm{LR}(A)/L)_{ab} \rightarrow p$-$\rm{LR}(A)/L$ the
forgetful functor. An object $M \in(p$-$\rm{LR}(A)/L)_{ab}$ is
called a \textit{Beck} $L$-\textit{module}. Let $\emph{g} \in
p$-$\rm{LR}(A)/L$ and $M$ a Beck-$L$-module. The group
$Hom_{p-\rm{LR}(A)/L}(\emph{g},I_{L}(M))$ is called the
\textit{group of Beck derivations} of $g$ by $M$.
\end{definition}

\begin{notation}
Let $M \xrightarrow{\mu} L$ be a Beck $L$-module. We denote by
$\bar{M}:=ker\,\mu$ and by $p_{\bar{M}}$ the restriction of the
$p$-map $p_{M}$ of $M$ to $\bar{M}$.
\end{notation}
\begin{theorem}
Let $L\in p$-$\rm{LR}(A)$ and $M \xrightarrow{\mu} L$ be Beck
$L$-module. Then, $\bar{M}\rtimes_{p_{\bar{M}}} L$ is defined and
endowed with the structure of a restricted Lie-Rinehart algebra.
Moreover, there is an isomorphism of restricted Lie-Rinehart
algebras:
$$\bar{M}\rtimes_{p_{\bar{M}}} L\simeq M$$
\end{theorem}

\begin{proof}
Since $M \xrightarrow{\mu} L$ is an abelian group object in
$(p$-$\rm{LR}(A)/L)_{ab}$ a fortiori is an abelian group object in
$(Lie/L)_{ab}$, where $Lie/L$ denotes the slice category of Lie
algebras over $L$. It is well known that the category of abelian
group objects $(Lie/L)_{ab}$ is equivalent to the category of Lie
$L$-modules. In particular if $z: L\rightarrow M$ denotes the zero
map for the structure of group object we have a split extension in
the category of Lie algebras:
$$0\rightarrow \bar{M}\rightarrow M\xrightarrow[\xleftarrow{z}]{\mu} L\rightarrow
0$$

Moreover, there is an isomorphism of Lie algebras $\psi:
\bar{M}\oplus L\simeq M$ given by $\psi(\bar{m}+X):=\bar{m}+z(X)$,
for all $\bar{m}\in \bar{M}$ and $X\in L$. Since $z$ is a restricted
Lie homomorphism we have:

\begin{align*}
X^{[p_{L}]}\bar{m} &=[z(X^{[p_{L}]}),\bar{m}]\\
                   &=[z(X)^{[p_{M}]},\bar{m}]\\
                   &=[\underbrace{z(X),\cdots
                   [z(X),[z(X)}_{p},\bar{m}]]\cdots]
\end{align*}
Therefore, $\bar{M}$ is endowed with the structure of restricted
$L$-module. Besides $\bar{M}$ is abelian thus, $$[z(X),
\bar{m}^{[p_{\bar{M}}]}]=[z(X),\underbrace{\bar{m}]...,\bar{m}}_{p}]..]=0$$
Therefore, $p_{\bar{M}}: \bar{M}\rightarrow \bar{M}^{L}$ and
$\bar{M}\rtimes_{p_{\bar{M}}}L$ is defined. Since $\bar{M}$ is
abelian
\begin{align*}
(\psi(X+\bar{m}))^{p_{M}}
&=z(X)^{p_{M}}+\bar{m}^{p_{M}}+\sum_{i=1}^{i=p-1}s_{i}(z(X),\bar{m})\\
&=z(X)^{p_{M}}+\bar{m}^{p_{M}}+ad_{z(X)}^{p-1}(\bar{m})
\end{align*}
Therefore, the isomorphism $\psi$ is an isomorphism of a restricted
Lie algebras $$\psi: \bar{M}\rtimes_{p_{\bar{M}}} L\simeq M$$
 The $k$-module $\bar{M}\rtimes_{p_{\bar{M}}} L$ has the structure
 of an $A$-module given by the formula $$a(\bar{m}+X):=a\bar{m}+aX,\;\;a\in A$$ and
$\psi$ is an $A$-module homomorphism. We easily see that,
$\bar{M}\rtimes_{p_{\bar{M}}} L$ is endowed with the structure of a
Lie-Rinehart algebra with anchor map $\alpha:
\bar{M}\rtimes_{p_{\bar{M}}} L \rightarrow Der_{k}(A) $ given by
$$\alpha (\bar{m}+X)(a):=X(a)=z(X)(a)$$
In this way, $\psi$ becomes a Lie-Rinehart isomorphism. Besides, we
have
\begin{align*}
\psi((a(X+\bar{m}))^{[p]}) &=(\psi(a(X+\bar{m})))^{[p]}\\
                            &=(a(\psi(X+\bar{m})))^{[p]}\\
                            &=a^{p}(\psi(X+\bar{m}))^{[p]}+(a\psi (X+\bar{m}))^{p-1}(a)\psi( X+\bar{m})\\
                            &=a^{p}(\psi(X+\bar{m})^{[p]}))+(\psi(a(X+\bar{m}))^{p-1}(a)(\psi(X+\bar{m})))\\
                            &=\psi(a^{p}(X+\bar{m})^{[p]})+(a(z(X)+\bar{m}))^{p-1}(a)\psi(X+\bar{m}))\\
                            &=\psi(a^{p}(X+\bar{m})^{[p]})+\psi((a(z(X)+\bar{m}))^{p-1}(a)(X+\bar{m}))\\
                            &=\psi(a^{p}(X+\bar{m})^{[p]}+(a(X+\bar{m}))^{p-1}(a)(X+\bar{m}))\\
\end{align*}
Therefore, $M \rtimes_{p_{\bar{M}}} L$ is a restricted Lie-Rinehart
algebra and $\psi: \bar{M}\rtimes_{p_{\bar{M}}} L \simeq M$ is a
restricted Lie-Rinehart isomorphism.
\end{proof}

\begin{lemma}
Let $L$ be a Lie-Rinehart algebra over $A$. Then, the following
relation holds
$$(aX)^{p-1}(ab)=a^{p}X^{p-1}(b)+(aX)^{p-1}(a)b$$
for all $X\in p$-$\rm{LR}(A)$ and $a,b\in A$.
\end{lemma}

\begin{proof}
We consider the restricted enveloping algebra $U(A,L)$ of the
Lie-Rinehart algebra $L$ over $A$. By Hochschild's Lemma $1$, we get
in $U(A,L)$ that:
\begin{align*}
(a(b+i_{L}(X)))^{p} &=a^{p}(b+i_{L}(X))^{p}+(aX))^{p-1}(a)(b+i_{L}(X))\\
             &=a^{p}(b^{p}+i_{L}(X)^{p}+X^{p-1}(b))+(aX)^{p-1}(a)b+(aX)^{p-1}(a)i_{L}(X)
\end{align*}
Besides,
\begin{align*}
(a(b+i_{L}(X)))^{p} &=(ab+ai_{L}(X))^{p}\\
             &=((ab)^{p}+ai_{L}(X))^{p}+(aX)^{p-1}(ab)\\
             &=a^{p}b^{p}+a^{p}i_{L}(X)^{p}+(aX)^{p-1}(a)i_{L}(X)+(aX)^{p-1}(ab)
\end{align*}
Therefore, we get:
$$(aX)^{p-1}(ab)=a^{p}X^{p-1}(b)+(aX)^{p-1}(a)b$$
for all $X\in L$ and $a,b\in A$.
\end{proof}

Let $L$ be a Lie-Rinehart algebra over $A$. Let $M$ be a
Lie-Rinehart $(A-L)$-module. We consider the commutative algebra
semi-direct product $A\oplus M$ with product given by:
$$(a\oplus m)(a'\oplus m'):=aa'\oplus (am'+a'm)$$
There is an anchor map $\alpha: L\rightarrow Der(A\oplus M)$ given
by
$$\alpha(X)(a\oplus m):=X(a)\oplus X(m)$$ for all $X\in L, a\in A$
and $m\in M$. Moreover, $L$ becomes an $A\oplus M$-module via the
action $(a\oplus m)X:=aX$. Then, we easily see that, $L$ becomes a
Lie-Rinehart algebra over $A\oplus M$. Therefore, by Lemma $3.4$
above we get:
\begin{equation}
(aX)^{p-1}(am)=a^{p}X^{p-1}(m)+(aX)^{p-1}(a)m
\end{equation}
where $X\in L$ and $a\in A, m\in M$.

\begin{proposition}
Let $L$ be a restricted Lie-Rinehart algebra over $A$ and $\bar{M}$
a restricted Lie-Rinehart module. Then, $\bar{M}\rtimes L$ is
endowed with the structure of restricted Lie-Rinehart algebra.
\end{proposition}

\begin{proof}
We observe that $\bar{M}\rtimes L$ is an $A$-module via the action
$a(\bar{m}+X):=a\bar{m}+aX$ where $a\in A,\;\bar{m}\in\bar{M}$ and
$X\in L$. Moreover, there is anchor map $\alpha:\bar{M}\rtimes L
\rightarrow Der(A)$ given by $\alpha((\bar{m}+X))(a):=X(a)$. We
easily see that $\bar{M}\rtimes L$ is endowed with the structure of
Lie-Rinehart algebra over $A$. Besides, by the relation $(3.1)$
above we get
\begin{align*}
(a(\bar{m}+X))^{[p]} &=(aX)^{[p]}+(aX)^{p-1}(a\bar{m})\\
                     &=a^{p}X^{[p]}+(aX)^{p-1}(a)X+a^{p}X^{p-1}(\bar{m})+(aX)^{p-1}(a)\bar{m}\\
                     &=a^{p}(\bar{m}+X)^{[p]}+(a(\bar{m}+X))^{p-1}(a)(\bar{m}+X)
\end{align*}
Thus, $\bar{M}\rtimes L$ is a restricted Lie-Rinehart algebra.
\end{proof}

Let $A[P]$ be the polynomial ring given by
$$A[P]:=\{\sum_{i=0}^{i=n}a_{i}P^{i},: Pa=a^{p}P\;\; \text{for all}\;\; a_{i},a \in A\}$$

We consider the ring $W(A,L)$ which as an $A$-module is given by
$$W(A,L):=A[P]\otimes_{A} U_{p}(A,L)$$ and such that $A[P]\rightarrow
W(A,L)$ and $U_{p}(A,L)\rightarrow W(A,L)$ are $A$-algebra
homomorphisms and the multiplication is such that
\begin{align}
(P\otimes 1)(1\otimes i_{L}(X)):&=P\otimes i_{L}(X)\\
(1\otimes i_{L}(X))(P\otimes 1):&=0
\end{align}

\begin{proposition}
Let $L$ be a restricted Lie-Rinehart algebra over $A$. Then, the
category of Beck $L$-modules is equivalent to the category of
$W(A,L)$-modules.
\end{proposition}
\begin{proof}
Let $M$ be a $W(A,L)$-module. Using the homomorphism
$U_{p}(A,L)\rightarrow W(A,L)$ we can see $M$ as a
$U_{p}(A,L)$-module which we denote by $\bar{M}$. Besides, the
$A[P]$ action endows $\bar{M}$ with a $p$-semi-linear map
$p_{\bar{M}}: \bar{M} \rightarrow \bar{M}^{L}$ given by
$$p_{\bar{M}}(\bar{m}):=P\bar{m},\;\;\bar{m}\in \bar{M}$$ From
Proposition $3.5$ and Remark $1.4$ we see that
$\bar{M}\rtimes_{p_{\bar{M}}}L$ is a restricted Lie-Rinehart
algebra. Thus, by Theorem $3.3$, any $W(A,L)$-module $M$ is
associated to the Beck $L$-module $\bar{M}\rtimes_{p_{\bar{M}}}L$.
Conversely, let $M$ be a Beck $L$-module, then by Theorem $3.3$ we
have that $M\simeq \bar{M}\rtimes_{p_{\bar{M}}} L$. Moreover, we
observe that $\bar{M}$ is a $W(A,L)$-module. Thus, we have an
equivalence of categories.
\end{proof}

\subsection{Beck derivations}

Let $g$ be a restricted Lie-Rinehart algebra and
$M=\bar{M}\rtimes_{p_{\bar{M}}}L$ be a Beck $g$-module. Then,
$\bar{M}$ is a Lie-Rinehart $g$-module. We denote by
$Der(g,\bar{M})$ the group of Lie-Rinehart derivations. We recall
that a $k$-Lie algebra derivation $d: g\rightarrow \bar{M}$ is
called Lie Rinehart if $d$ is $A$-linear. The group of \textit{Beck
derivations} is defined as follows:
$$Der_{p}(g,M):=\{d\in Der_{A}(g,\bar{M}) : d(X^{[p]})=\underbrace{X\cdots
X}_{p-1}d(X)+(d(X))^{[p_{\bar{M}}]},\;\;X\in g\}$$
 We note that $Der_{p}(g,M)$ is a group under the addition since $p_{\bar{M}}$ is a $p$-semi-linear
map.

\begin{proposition}
Let $L$ be a restricted Lie-Rinehart algebra over $A$ and $g\in
p$-$\rm{LR}(A)/L$. If $M$ is a Beck $L$-module, then we have the
following isomorphism
$$Hom_{p-\rm{LR}(A)/L}(g,\bar{M}\rtimes_{p_{\bar{M}}} L )\simeq
Der_{p}(g,M)$$
\end{proposition}

\begin{proof}
Let $f: g \rightarrow \bar{M}\rtimes_{p_{\bar{M}}} L$ and $\pi:
\bar{M}\rtimes_{p_{\bar{M}}} L \rightarrow \bar{M}$ be the canonical
projection. We easily see that, $d_{f}:=\pi f$ is a Beck derivation
and therefore, is defined a map $\Phi: f \mapsto d_{f}$. Moreover,
let $d: g \rightarrow \bar{M}$ be a Beck derivation we consider the
map $f_{d}: g\rightarrow \bar{M}\rtimes_{p_{\bar{M}}} L$ given by
$f_{d}(X):=d(X)+\gamma(X)$, where $X\in g$ and $\gamma: g\rightarrow
L$ is the structural map. The maps $\Psi: d \mapsto f_{d}$ and
$\Phi: f\mapsto d_{f}$ are inverse to each other.
\end{proof}

\section{Quillen-Barr-Beck cohomology for restricted Lie-Rinehart algebras}

There is a general theory of cohomology for univeral algebras due to
Quillen-Barr-Beck (see \cite{Qui}, \cite{Bar}, \cite{BB}). Moreover,
Quillen in \cite{Qui2} proves that the cohomology theory for
universal algebras defined in \cite{Qui}, is a special case of the
general definition of sheave cohomology due to Grothendieck.
Following the general scheme of Quillen-Barr-Beck cohomology theory
we define cohomology groups for the category of restricted
Lie-Rinehart algebras. By Proposition $2.8$ there is a functor
$$F : \mathrm{Vect}/Der(A) \rightarrow p-\rm{LR}(A)$$ left adjoint
to the functor
$$\mathcal{V}: p-\mathrm{LR}(A) \rightarrow \mathrm{Vect}/Der(A)$$

The adjoint pair $(F,\mathcal{V})$ induce a cotriple
$\mathbf{G}=(G_{*}, \epsilon, \delta)$ such that
$G_{*}:=F\mathcal{V}$. Thus, for all $L\in p$-$\rm{LR}(A)$ we obtain
a simplicial resolution $G_{*}L\rightarrow L$ known as cotriple
resolution or Godement resolution (see \cite{God}, \cite{BB}).

\begin{definition}
Let $L$ be a restricted Lie-Rinehart algebra and $M$ a Beck
$L$-module. Then, for  $n\in \mathbb{N}^{*}$, Quillen-Barr-Beck
cohomology groups are defined by the following formula
$$H_{p-\rm{LR}}^{n}(L,M):=H^{n}(Hom_{p-\rm{LR}(A)/L}(G_{*}L,M))$$
where in the right hand side of the formula $H$ denotes the
cohomology of a cosimplicial object.
\end{definition}

\begin{remark}
We observe that in degree $0$ we have
$H_{p-\rm{LR}}^{0}(L,M)=Der_{p}(L,M)$.
\end{remark}

\subsection{Cohomology in degree $1$ and extensions}

A principal bundle gives rise to \textit{Atiyah sequence} introduced
by Atiyah in \cite{Ati} (see for details \cite{Hue2}). The algebraic
analogue of "Atiyah sequence" is an extension of Lie-Rinehart
algebras. In this subsection we consider extensions of restricted
Lie-Rinehart algebras. We prove that the Quillen-Barr-Beck
cohomology in degree one classifies restricted Lie-Rinehart
extensions.

\begin{definition}
An extension $(e)$ of restricted Lie-Rinehart algebras (of $L$ by
$M$) is a short exact sequence of restricted Lie-Rinehart algebras
$$0\rightarrow M\rightarrow E \rightarrow L\rightarrow 0\;\;\;\;\; (e)$$
such that $[M,M]=0$.
\end{definition}

\begin{remark}

Let $L$ be a restricted Lie-Rinehart algebra over $A$ and $(e)$  an
extension of $L$ by $M$. Since $M$ is abelian a section $s:
L\rightarrow E$ defines a restricted Lie $L$-action on $M$.
Moreover, we easily see that $M$ becomes a $U_{p}(A,L)$-module.
Besides, there is an $A[P]$ action on $M$ such that
$Pm:=m^{[p_{M}]},\;\;m\in M$. Thus, we obtain that $M$ is endowed
with the structure of a $W(A,L)$-module.
\end{remark}

Two restricted Lie-Rinehart extensions $(e),(e')$ of $L$ by $M$ are
called equivalent if there is a restricted Lie-Rinehart isomorphism
$f: E\rightarrow E'$ such that the following diagram commutes

\[
\begin{CD}
0 @>>> M  @>>>  E @>>>  L @>>> 0   \\
 @.    @|   @VfVV @|  \\
 0 @>>> M  @>>>  E' @>>> L @>>> 0
\end{CD}
\]

We denote the set of equivalent classes by $Ext_{p}(L,M)$.

\textbf{Baer sum of restricted Lie-Rinehart algebras}. Let
$(e),(e')$ be two Lie algebra extensions of $L$ by $M$
$$0\rightarrow M \rightarrow E \xrightarrow{f} L\rightarrow 0$$ and
$$0\rightarrow M \rightarrow E' \xrightarrow{f'} L\rightarrow 0$$
Let $E\times_{L} E'=\{(e,e'),:f(e)=f'(e')\}$ be the pullback in
$Lie$. We denote by $I$ the ideal
$$I:=<\{(m,0)-(0,m'),\;\; m,m'\in M \}>$$ and we consider the Lie
algebra
$$Y:=E\times_{L} E'/I$$
The Baer sum of $(e)$ and $(e')$ is the extension of Lie algebras of
$L$ by $M$
$$0\rightarrow M \xrightarrow{\iota} Y \xrightarrow{\psi} L\rightarrow 0$$
where  $\iota(m):=\overline{(m,0)}$ and the last one
$\psi\big(\overline{(e,e')}\big):=f(e)=f(e')$. The set
$Ext_{Lie}(L,M)$ of equivalence classes of extensions of  Lie
algebras of $L$ by $M$ endowed with the operation of Baer sum is a
group. Moreover, let $(e),(e')$ be restricted Lie-Rinehart
extensions of $L$ by $M$. Then, $E\times_{L} E'$ is a restricted Lie
algebra with $p$-map given by $(e,e')^{[p]}=(e^{[p]},e'^{[p]})$. Let
$a\in A$ and $(X,Y)\in E\times_{L} E'$. Then, there is an action of
$A$ on $E\times_{L} E'$ given by $a(X,Y):=(aX,aY)$. Besides, there
is an action $E\times_{L} E' \rightarrow Der(A)$ given by
$(X,Y)(a):=X(a)=Y(a)$. The above actions endow $E\times_{L} E'$ with
the structure of a restricted Lie-Rinehart algebra. Moreover, the
ideal $I$ is a restricted Lie-Rinehart ideal. Thus, $Y$ is a
restricted Lie-Rinehart algebra. The Baer sum endows $Ext_{p}(L,M)$,
with the structure of a group and $Ext_{p}(L,M)$ is a subgroup of
$Ext_{Lie}(L,M)$.

\begin{theorem}
Let $L$ be a restricted Lie-Rinehart algebra over $A$ and $M$ a Beck
$L$-module. There is a bijection
\begin{equation}
H_{p-\rm{LR}}^{1}(L,M) \simeq Ext_{p}(L,M)
\end{equation}
\end{theorem}

\begin{proof}
Duskin in \cite{Du} develops the theory of torsors and gives an
interpretation of cotriple cohomology. There is a bijection between
the first cohomology group $H_{G}^{1}(L,M)$ and the set
$Tors_{p-LR}(L,M)$ of the isomorphism classes of objects $E\in
p$-$\rm{LR}(A)/L$ which are torsors for the abelian group object
$M$. An object $E\in p$-$\rm{LR}(A)/L$ is $M$ torsor, if
$E\rightarrow L$ is an epimorphism and there is a restricted
Lie-Rinehart morphism
$$\omega: (\bar{M}\rtimes_{p_{\bar{M}}} L)\times_{L} E\rightarrow E$$
such that the map $$(\omega, \pi):(\bar{M}\rtimes_{p_{\bar{M}}}
L)\times_{L} E\rightarrow E\times_{L} E$$ where $\pi$ denotes
projection, is a restricted Lie-Rinehart isomorphism.

Let $E\in Ext_{p}(L,M)$, then we have a short exact sequence in
$p$-$\rm{LR}(A)$
$$0\rightarrow M\rightarrow E\xrightarrow{f} L\rightarrow 0$$
such that $[M,M]=0$ and the induced $W(A,L)$-action on $M$, recovers
the given $W(A,L)$-action on $M$. We define a map $$\omega:
(\bar{M}\rtimes_{p_{\bar{M}}} L)\times_{L} E\rightarrow E$$ given by
$\omega (m+f(e),e):=m+e$, for all $m\in \bar{M}$ and $e\in E$. We
easily see that, $\omega$ is a Lie-Rinehart homomorphism. Moreover,
we have
\begin{align*}
\omega((m+f(e),e)^{[p]}) &=\omega((m+f(e))^{[p]},e^{[p]})\\
                 &\omega(f(e)^{p-1}m+m^{[p_{\bar{M}}]}+(f(e))^{[p]},e^{[p]})\\
                 &=\omega(e^{p-1}m+ m^{[p_{\bar{M}}]}+f(e^{[p]}),e^{[p]})\\
                 &=e^{p-1}m+ m^{[p_{\bar{M}}]}+e^{[p]}\\
                 &=(m+e)^{[p]}\\
                 &= (\omega(m+f(e),e))^{[p]}
\end{align*}
Thus, we get that $\omega$ is a restricted Lie-Rinehart homomorphism. It is easy to check that $$(\omega,
\pi):(\bar{M}\rtimes_{p_{\bar{M}}} L)\times_{L} E\rightarrow
E\times_{L} E$$ is a restricted Lie-Rinehart isomorphism.
Therefore, $E\rightarrow L$ is torsor for $M$.

Conversely, let $E\in p$-$\rm{LR}(A)/L$ be an object torsor for
$\bar{M}\rtimes_{p_{\bar{M}}} L\rightarrow L$. Then, the structural
map $f: E\rightarrow L$, is a restricted Lie-Rinehart epimorphism.
If $K:=ker\,f$, then there is an injection $i: K \hookrightarrow
E\times_{L} E$ with $i(k):=(k,0)$. Besides, there is an injection
$j: \bar{M} \hookrightarrow (\bar{M}\rtimes_{p_{\bar{M}}}
L)\times_{L}E$ given by $j(\bar{m}):=((\bar{m},0),0)$. The
restriction of the restricted Lie-Rinehaert isomorphism
$(\omega,\pi)$ on $\bar{M}$ implies an isomorphism of restricted
Lie-Rinehart algebras $(A,\bar{M},(-)^{[p_{\bar{M}}]})\simeq
(A,K,(-)^{[p_{k}]})$. Moreover, let $X\in L,\;\;m\in \bar{M}$ and
$e\in E$ such that $f(e)=X$. Then, we have
\begin{align*}
\omega(Xm,0) &=\omega ([(X,e),(m,0)])  \\
                 &=[\omega(X,e),\omega(m,0)]
\end{align*}
Besides, we have $$(\omega,\pi)((X,e),e)=(\omega(X,e),e)$$ it
follows that $f(\omega(X,e))=X$. Therefore, $\bar{M}$ and $K$ are
isomorphic as restricted Lie $L$-modules. Since $\omega$ is a
restricted Lie-Rinehart homomorphism, it follows that, $\bar{M}$ and
$K$ are isomorphic as $U_{p}(A,L)$-modules. Thus, we get an
extension of restricted Lie-Rinehart algebras
 $$0\rightarrow M\rightarrow E
\rightarrow L\rightarrow 0$$ Therefore, there is a bijection of sets
$H_{p-\rm{LR}}^{1}(L,M) \simeq Ext_{p}(L,M)$.
\end{proof}

\begin{remark}
We note that, in terms of Quillen-Barr-Beck cohomology, there is a
shift in the dimension, in compare to the classical notation,
concerning classification theorems of extensions.
\end{remark}

If $L$ is a Lie algebra and $M$ a $\mathcal{U}(L)$-module, then
cotriple cohomology groups
$H^{*}_{Lie}(L,M):=H^{*}(Hom_{Lie}(G_{Lie}(L),M))$ are defined (see
\cite{Bar}). The cotriple which is used is given by
$G_{Lie}:=L\mathcal{F}$ where $L: Vect\rightarrow Lie$ is the free
Lie algebra functor left adjoint to the forgetful functor
$\mathcal{F}: Lie \rightarrow Vect$.

\begin{corollary}
There is an isomorphism of groups
$$H_{p-\rm{LR}}^{1}(L,M) \simeq Ext_{p}(L,M)$$
\end{corollary}

\begin{proof}
Duskin in \cite{Du},\cite{Du2} and Glen in \cite{Glen}, develop the
theory of torsors. In particular, is proved (see Section $4$ in
\cite{Du2} and \cite{Glen}) that the set of torsors can be endowed
with the structure of a group. It follows from the general theory
(see Section 5 in \cite{Du2} and \cite{Glen}) that, there is a group
isomorphism
$$H^{1}_{p-\rm{LR}}(L,M) \simeq Tors_{p-\rm{LR}}(L,M)$$ and
respectively for the category Lie algebras $$H^{1}_{Lie}(L,M) \simeq
Tors_{Lie}(L,M)$$

Besides, the category of Lie algebras is a category of interest in
sense of Orzech (see \cite{Or}). Therefore by a general result of
Vale in \cite{Vale}, we obtain a group isomorphism
$Ext_{Lie}(L,M)\simeq Tors_{Lie}(L,M)$ . If $L\in p-\rm{LR}(A)$ and
$M$ a $W(L)$-module, then there is a natural embedding
$H^{1}_{p-\rm{LR}}(L,M)\hookrightarrow H^{1}_{Lie}(L,M)$.

Thus, we have the following commutative diagram

\xymatrix{
H^{1}_{p-\rm{LR}}(L,M) \ar@{^{(}->}[d]  \ar[r]^{\backsim}  & Tors_{p-\rm{LR}}(L,M)  \ar@{^{(}->}[d]  \ar[r]  & Ext_{p}(L,M) \ar@{^{(}->}[d]    \\
H^{1}_{Lie}(L,M)  \ar[r]^{\backsim} &  Tors_{Lie}(L,M)
\ar[r]^{\backsim} & Ext_{Lie}(L,M)}

 Therefore, the bijection $(4.1)$
$$H_{p-\rm{LR}}^{1}(L,M) \simeq Ext_{p}(L,M)$$ is an isomorphism of
groups.
\end{proof}

\section{Brauer group and cohomology}

The problem of classification of finite-dimensional central simple
algebras is related to the notion of Brauer group. The theory of
Brauer groups has strong ties with number theory, algebraic geometry
(see \cite{Ser2}, Local fields) and algebraic $k$-theory \cite{Ker}.
In connection with Galois theory we have that if $E/F$ is a Galois
extension then the relative Brauer group $\mathcal{B}_{F}^{E}$ is
isomorphic with the group of equivalence classes of group extensions
of the Galois group $Gal(E/F)$ by the multiplicative group $F^{*}$
of $F$. Moreover, we have an isomorphism of groups relating the
Galois cohomology and the Brauer group:
\begin{equation}
\mathcal{B}_{F}^{E}\simeq H^{2}(Gal(E/F),F^{*})
\end{equation}

In the context of purely inseparable extension of exponent $1$ there
is a Galois theory due to Jacobson (see \cite{Jac2}). The role of
Galois group is now played by the group of derivations. In
particular, let $K/k$ be a finite purely inseparable extension of
exponent $1$ and $A$ a finite dimensional algebra with center $k$
which contains $K$ as a maximal commutative subring. Let $S:=\{s\in
A,| D_{s}(K)\subset K\}$, where $D_{s}$ denotes the derivation
$D_{s}(a):=sa-as$. Then, Hochschild in \cite{Hoch2} considers
particular class of restricted Lie algebra extensions  of
$Der_{k}(K)$ by $K$
$$0\rightarrow K\rightarrow S\rightarrow Der_{k}(K)\rightarrow 0$$
called \textit{regular extensions}. A regular extension is nothing
more than a restricted Lie-Rinehart extension of the restricted
Lie-Rinehart algebra $Der_{k}(K)$ by $K$. We note that $K$ is an
abelian restricted Lie algebra over $k$ with its natural $p$-map
(see \cite{Hoch2} for details). Moreover, Hochschild in \cite{Hoch2}
proves that there is an isomorphism between the Brauer group
$\mathcal{B}_{k}^{K}$ and the set of equivalence classes of regular
restricted Lie algebra extensions of $Der_{k}(K)$ by $K$.

As motivation to carry out this research was to establish an
isomorphism in terms of cohomology of restricted Lie-Rinehart
algebras analogue to the isomorphism $(5.1)$ in terms of Galois
cohomology. Hochschild in \cite{Hoch} defined cohomology groups
$H_{Hoch}^{*}:=H^{*}(L,M)$ where $L\in RLie$ and $M$ a restricted
Lie module. Since the second cohomology group $H^{2}_{Hoch}(L,M)$
classifies extensions where $M$ is strongly abelian, i.e such that
$M^{[p]}=0$ we do not have an isomorphism of the Brauer group
$\mathcal{B}_{k}^{K}$ with a sub group of
$H^{2}_{Hoch}(Der_{k}(K),K)$. Nevertheless, using the
Quillen-Barr-Beck cohomology for $p$-$\rm{LR}(A)$ we can establish
the analogue isomorphism to $(5.1)$.

\begin{theorem}
Let $K/k$ be a finite purely inseparable extension of exponent $1$.
Then, we have the following isomorphism of groups
$$\mathcal{B}_{k}^{K}\simeq H_{p-\rm{LR}}^{1}(Der_{k}(K),K)$$
\end{theorem}

\begin{proof}
The proof follows from Corollary $(4.7)$ and the isomorphism proved
by Hochschild in \cite{Hoch2} between the Brauer group
$\mathcal{B}_{k}^{K}$ and the group of equivalence classes of
\textit{regular} extensions of $Der_{k}(K)$ by $K$.
\end{proof}

\begin{remark}
Given a commutative ring $C$ and a commutative $C$-algebra
$\mathcal{A}$, Amutsur in \cite{Am} defined cohomology groups
$H_{Am}^{*}(\mathcal{A})$. Let $K$ be a purely inseparable extension
field of $k$ of exponent $1$. Then, Rosenberg and Zelinsky in
Section $4$ in \cite{Ros} exhibit an isomorphism between
$H_{Am}^{2}(K)$ and the Brauer group $\mathcal{B}_{k}^{K}$.
Therefore, by the above theorem we get an isomorphism
$$H_{Am}^{2}(K) \simeq H_{p-\rm{LR}}^{1}(Der_{k}(K),K)$$
\end{remark}

\bibliographystyle{plain}

\end{document}